\documentclass[reqno]{amsart}
\usepackage{graphicx}
\usepackage{geometry}
\usepackage[leqno]{amsmath}
\usepackage{amscd}
\usepackage{amsthm}
\usepackage{amsfonts}
\usepackage{amssymb}
\usepackage[colorlinks=true]{hyperref}
\usepackage{xcolor}
\usepackage{mathrsfs}
\newtheorem{theorem}{Theorem}[section]

\newtheorem{corollary}[theorem]{Corollary}
\newtheorem{proposition}[theorem]{Proposition}
\newtheorem{Lemma}[theorem]{Lemma}

\theoremstyle{definition}
\newtheorem{definition}[theorem]{Definition}
\newtheorem{example}[theorem]{Example}

\newtheorem{remark}[theorem]{Remark}
\makeatletter
\newcommand{\leqnomode}{\tagsleft@true}
\newcommand{\reqnomode}{\tagsleft@false}
\makeatother

\DeclareMathOperator{\End}{End}
\DeclareMathOperator{\Aut}{Aut}
\DeclareMathOperator{\Ab}{Ab}
\DeclareMathOperator{\Map}{Map}

\newcommand{\gen}[1]{\langle #1 \rangle}

\newcommand{\id}{\mathrm{id}}
\newcommand{\G}[1]{(G,\circ_{#1})}
\newcommand{\Bee}[2]{(G,\circ_{#1},\circ_{#2})}

\numberwithin{equation}{section}

\begin{document}
\newtheorem*{thm}{Theorem}

\title{Abelian maps, brace blocks, and solutions to the {Y}ang-{B}axter equation}
\author{Alan Koch}
\address{Department of Mathematics, Agnes Scott College, 141 E. College Ave., Decatur, GA\ 30030 USA \\akoch@agnesscott.edu}
\date{\today       }

\begin{abstract} Let $G$ be a finite nonabelian group. We show how an endomorphism of $G$ with abelian image gives rise to a family of binary operations $\{\circ_n: n\in \mathbb Z^{\ge 0}\}$ on $G$ such that $(G,\circ_m,\circ_n)$ is a skew left brace for all $m,n\ge 0$. A brace block gives rise to a number of non-degenerate set-theoretic solutions to the Yang-Baxter equation. We give examples showing that the number of solutions obtained can be arbitrarily large.
	\end{abstract}

\maketitle

\section{Introduction}
Rump \cite{Rump07} introduced the notion of {\it braces} to construct involutive, non-degenerate set-theoretic solutions to the {\it Yang-Baxter equation}, an equation whose solutions have numerous, well-known applications (see, e.g., \cite{Chen12,Jimbo89a, Jimbo89b, Jimbo94, LambeRadford97}). This construction was then generalized to {\it skew left braces} in \cite{GuarnieriVendramin17} to provide solutions which are not necessarily involutive. In \cite{Koch20} we show how one can construct many examples of skew left braces using what we call {\it abelian maps}, i.e., endomorphisms of a finite group $G=(G,\cdot)$ with abelian image; our description also allows us to easily generate a pair of set-theoretic solutions to the YBE if $G$ is nonabelian.

Here, we significantly improve upon our abelian map construction. Namely, we prove that a single abelian map $\psi: G\to G$ gives a family of skew left braces; we call this family a {\it brace block}. A brace block consists of $G$ together with a collection $\{\circ_n\}$ of binary operations such that $(G,\circ_m,\circ_n)$ is a skew left brace for all $m,n$. 

The skew left braces we construct are examples of {\it bi-skew braces}.  In \cite{Childs19} Childs introduces bi-skew braces: a bi-skew brace is a skew left brace which is also a skew left brace if the operations are reversed. The skew left braces in our block are evidently bi-skew because both $\Bee{m}{n}$ and $\Bee{n}{m}$ are skew left braces for all $m,n$. We will tend to view our objects as skew left as opposed to bi-skew with the goal of simplifying the statements of our results. 

Starting with an abelian map $\psi:G\to G$ we generate a family of abelian maps $\{\psi_n: n\in\mathbb Z^{\ge 0}\}$ on $G$. We then use \cite[Th. 1]{Koch20} to construct the group $\G{n}$ and produce the skew left brace $(G,\cdot,\circ_n)$. It turns out that $\Bee{m}{n}$ is a skew left brace for every choice of $m,n\ge 0$, hence we get a brace block. Each skew left brace $\Bee{m}{n}$ provides a non-degenerate set-theoretic solution $R$ to the YBE; if $\G{m}$ is abelian this solution is involutive, otherwise we obtain a second solution $R'$, inverse in the sense that $RR'=R'R=\id$. 

The paper is structured as follows. After providing background definitions and setting notation, we construct our abelian maps $\psi_n$ for $n\ge 0$ and give their properties. We show how $\psi_n$ gives us a binary operation $\circ_n$ on $G$. We then introduce the concept of brace blocks and arrive at main result, Theorem \ref{main}, which proves that $(G,\circ_m,\circ_n)$ is a skew left brace for each $m,n\ge 0$, and hence the set of all such braces forms a brace block. We then explicitly construct a pair of inverse solutions to the Yang-Baxter equation from each $\Bee{m}{n}$; as stated previously, the two solutions will coincide if and only if $\G{m}$ is abelian. The solutions coming from $\Bee{m}{n}$ will be entirely in terms of $\psi_m$, $\psi_n$, and the underlying group operation on $G$.  We then do some first examples of brace blocks, focusing on cases where $\psi:G\to G$ is fixed point free in the sense of \cite{Childs13}; these examples tend to produce a small number of solutions to the YBE. Finally, we examine a set of examples which allow us to produce large numbers of solutions; in fact, we will see that there is no bound to the number of solutions obtainable using abelian maps.

Throughout, $G$ is a nonabelian group: while the constructions to follow will work if $G$ is abelian, the brace block will produce only one brace, namely the trivial brace on $G$, making the block construction unnecessary. Since we will put multiple group structures on $G$, we will denote center of $G$ with respect to the operation $\ast$ by $Z(G,\ast)$ and write $Z(G)=Z(G,\cdot)$.

\section{Braces braces and the {Y}ang-{B}axter equation}
Here we shall give a short overview of the connection between different types of braces and the Yang-Baxter equation. Recall that a {\it set-theoretic solution to the Yang-Baxter equation} consists of a set $B$ together with a function $R:B\times B \to B\times B$ such that
\[(R\times \id)(\id \times R)(R\times \id) = (\id \times R)(R\times \id)(\id \times R):B\times B\times B \to B\times B\times B.\]
When $B$ is understood we simply denote the solution by $R$.
Given a solution $R$, write $R(x,y)=(f_y(x),g_x(y))$. We say $R$ is {\it non-degenerate} if both $f_y:B\to B$ and $g_x:B\to B$ are invertible. Additionally, $R$ is {\it involutive} if $R^2=\id$. An example of a non-degenerate, involutive solution is constructed by taking $B$ to be any set and defining $R(x,y)=(y,x)$ for all $x,y\in B$.

Set-theoretic solutions were suggested by Drinfeld \cite{Drinfeld92} as a way to obtain solutions to the Yang-Baxter equation on a vector space $V$, i.e., endomorphisms $V\otimes V\to V\otimes V$ satisfying a ``twisting'' condition analogous to the one above. Given a set-theoretic solution with underlying set $B$, the corresponding vector space solution arises by letting $V$ be the vector space with basis $B$. With this is mind, we will say that two (set-theoretic) solutions are {\it equivalent} if they induce the same vector space solution up to a choice of basis.

Herein, we will use ``solution'' to mean a non-degenerate, set-theoretic solution. In general, our solutions will not be involutive.

Skew left braces are a useful tool for finding solutions to the Yang-Baxter equation. While there have been numerous papers on skew left braces recently, we feel a quick definition is in order to set notation. A {\it skew left brace} is a triple $(B,\cdot,\circ)$ consisting of a set and two binary operations such that $(B,\cdot)$ and $(B,\circ)$ are groups and
\[
	x\circ(y\cdot z) = (x\circ y)\cdot x^{-1}\cdot (x\circ z)\\
\] 
 holds for all $x,y,z\in B$, where $x^{-1}\cdot x = 1_B$, the identity common to both operations. We refer to the condition above as the {\it brace relation}. We will follow the typical (but not universal) practice of writing $x\cdot y$ as $xy$ when no confusion will arise, and $\overline x$ for the inverse of $x$ in $(G,\circ)$. 
 
 A very simple example of a brace is $(G,\cdot,\cdot)$ where $(G,\cdot)$ is any group. We call this the {\it trivial brace on $G$}, or just the {\it trivial brace} for short.
 
 Skew left braces are generalizations of {\it left braces}, found in \cite{Rump07} where $(B,\cdot)$ is assumed to be abelian. For simplicity, we will refer to a skew left brace simply as a brace. 
 
Every brace $(B,\cdot,\circ)$ gives a solution to the Yang-Baxter equation \cite[Th. 3.1]{GuarnieriVendramin17}, namely
\begin{equation} R(x,y)=(x^{-1}(x\circ y),\overline{x^{-1}(x\circ y)}\circ x \circ y)\label{one},\;x,y\in B.\end{equation}
There is also the notion of an {\it opposite brace}, considered independently by the author and Truman in \cite{KochTruman20} and Rump in \cite{Rump19}. The opposite brace to $(B,\cdot,\circ)$, which is simply $(B,\cdot',\circ)$ with $x\cdot'y=y\cdot x$, provides an additional solution to the YBE:
\begin{equation}
	R'(x,y)=((x\circ y)x^{-1},\overline{(x\circ y)x^{-1}}\circ x \circ y)\label{two}.\\
\end{equation}
Note that $R=R'$ if and only if $(B,\cdot)$ is abelian. By \cite[Th. 4.1]{KochTruman20} we have that $RR'=R'R=\id$. 

We say that two braces $(B,\cdot,\circ)$ and $(B',\cdot',\circ')$ are {\it isomorphic} if there exists a function $B\to B$ which is an isomorphism $(B,\cdot)\to (B',\cdot')$ and also an isomorphism $(B,\circ)\to (B',\circ')$. Isomorphic braces will give equivalent solutions to the Yang-Baxter equation. 

\section{Endomorphisms from an abelian map}

The motivation for the construction to follow can be found in \cite{Koch20}, where an abelian map $\psi:(G,\cdot)\to (G,\cdot)$ gives rise to a brace $(G,\cdot,\circ)$. It is not hard to show that $\psi$ respects $\circ$ as well, and that $\psi(G)$ is an abelian subgroup of $(G,\circ)$. Viewing $\psi:(G,\circ)\to(G,\circ)$ then allows for the construction of another brace, say $(G,\circ,\diamond)$. This can then be repeated {it ad infinitum}, of course, generating a ``chain'' of braces. It turns out that any two operations generated in this manner can be put together in a brace as well, e.g., $(G,\cdot,\diamond)$. The simplest way to describe this chain of operations is through the abelian maps $\psi_n$ introduced below. 

Let $G$ be a group, written multiplicatively, and write $\Map(G)$ for the set of all functions $\phi:G\to G$. We write $1\in\Map(G)$ to denote the identity map and $0\in\Map(G)$ for the trivial map. For $\phi,\psi\in\Map(G)$ define 
\begin{align*}
	(\phi+\psi)(g)&=\phi(g)\psi(g)\\
	-\phi(g) &= \phi(g^{-1})\\
	\phi\psi(g) &= \phi(\psi(g)).
\end{align*}
Then $\Map(G)$ carries the structure of a right near-ring: $(\Map(G),+)$ is a group (necessarily nonabelian), the multiplication is associative and the right distributive law holds. We define $\phi^n$ in the usual way for $n\ge 0$.

We are interested primarily in the subset $\End(G)\subset \Map(G)$ consisting of endomorphisms of $G$, particularly the set $\Ab(G)\subset \End(G)$ of maps $\psi:G\to G$ with $\psi(G)$ abelian: we call such endomorphisms {\it abelian maps}. A crucial fact about abelian maps is that they are constant on conjugacy classes: we will see how this observation greatly simplifies some calculations.

Neither $\Ab(G)$ nor $\End(G)$ are, in general, closed under the group operation, though both contain the identity and $\Ab(G)$ contains its inverses.  Notice that if $\psi\in\Ab(G)$ then $\psi\phi\in\Ab(G)$ for all $\phi\in\End(G)$; in particular,  $\psi^n\in\Ab(G)$ for all $n\ge 0$.

We shall now introduce a collection of abelian maps that can be constructed from a single abelian map. Let $\psi\in\Ab(G)$. For $n\ge 0$, define
\[\psi_n = - (1-\psi)^n+1.\]
Note that $\psi_0=0$ and $\psi_1=\psi$. While the binomial formula does not work in $\Map(G)$ generally, we do have the following useful result. 
\begin{Lemma} Let $\psi\in\Ab(G)$. For $n\ge 1$ we have 
	\[\psi_n=\sum_{i=1}^n (-1)^{i-1}\binom ni \psi^{i} ,\] i.e.,
\begin{equation}\label{form}\psi_n(g) =\prod_{i=1}^{n}\psi^i\left(g^{(-1)^i\binom {n}i}\right)=\psi\left(g^{\binom n1}\right)\psi^2\left(g^{-\binom n2}\right)\cdots\psi^n\left(g^{(-1)^n\binom nn}\right),\;g\in G. \end{equation}
Thus, $\psi_n\in\Ab(G)$ for all $n\ge 0$.
\end{Lemma}

\begin{proof}

	First, we will prove that  
\begin{equation}\label{bin}(1-\psi)^n(g)=g\prod_{i=1}^{n}\psi^i\left(g^{(-1)^i\binom ni}\right)\end{equation}
holds for all $g\in G,\;n\ge 0$ by induction on $n$. It is clear that Equation \ref{bin} holds for $n=0$ as both sides reduce to $g$. Now assume Equation \ref{bin} holds for $n=k$. Then, since $\Map(G)$ has a right distributive law and $\psi$ is abelian,
\begin{align*}
	(1-\psi)^{k+1}(g)&=(1-\psi)(1-\psi)^{k}(g)\\
	&=(1-\psi)^{k}(g)-\psi((1-\psi)^{k}(g))\\
	&=g\prod_{i=1}^{k}\psi^i\left(g^{(-1)^i\binom {k}i}\right)\psi\left(g\prod_{i=1}^{k}\psi^i(g^{(-1)^i\binom {k}i})\right)^{-1}\\
	&=g\prod_{i=1}^{k}\psi^i\left(g^{(-1)^i\binom {k}i}\right)\cdot\left(\prod_{i=1}^{k}\psi^{i+1}(g^{(-1)^{i+1}\binom {k}i})\right)\psi(g^{-1})\\
	&=g\prod_{i=1}^{k}\psi^i\left(g^{(-1)^i\binom {k}i}\right)\cdot\left(\psi(g^{-1})\prod_{i=2}^{k+1}\psi^{i}(g^{(-1)^{i}\binom {k}{i-1}})\right)\\
	&=g\prod_{i=1}^{k}\psi^i\left(g^{(-1)^i\binom {k}i}\right)\cdot\left(\prod_{i=1}^{k+1}\psi^{i}(g^{(-1)^{i}\binom {k}{i-1}})\right)\\
	%&=\psi^0(g^{(-1)^0\binom {k}{0}})\left(\prod_{i=1}^{k}\psi^i\left(g^{(-1)^i\binom {k}i+(-1)^{i}\binom {k}{i-1}}\right)\right)\psi^{k+1}(g^{(-1)^{k+1}\binom {k}{k}})\\
	&=\left(\prod_{i=1}^{k}\psi^i\left(g^{(-1)^i\binom {k}i+(-1)^{i}\binom {k}{i-1}}\right)\right)\psi^{k+1}(g^{(-1)^{k+1}\binom {k+1}{k+1}})\\
	&=\prod_{i=1}^{k+1}\psi^i\left(g^{(-1)^i\binom {k+1}i}\right)
\end{align*}
and Equation \ref{bin} is true for all $n$.

Next,\[
	\psi_n(g)=(-(1-\psi)^n+1)(g)  = \left(g\prod_{i=1}^{n}\psi^i\left(g^{(-1)^i\binom {n}i}\right)\right)^{-1}g
=\prod_{i=1}^{n}\psi^i\left(g^{(-1)^i\binom {n}i}\right),\]
and so Equation \ref{form} holds. That $\psi_n$ is an endomorphism follows from the fact that each of $\psi^i$ are endomorphisms and that \[\psi^i(g^k)\psi^j(h^{\ell})=\psi(\psi^{i-1}(g^k)\psi^{j-1}(h^{\ell}))=\psi(\psi^{j-1}(h^{\ell})\psi^{i-1}(g^k))=\psi^j(h^{\ell})\psi^i(g^k)\]
for all $i,j,k,\ell$ with $i,j\ge 1$. Finally, $\psi_n(g)$ is a product of elements in $\psi(G)$, hence $\psi_n$ is abelian for $n\ge 1$, and since $\psi_0$ is trivial we get $\psi_0\in\Ab(G)$ and the lemma is proved.
\end{proof}

For future reference we make note of some useful properties.

\begin{Lemma}\label{lem} For $\psi\in\Ab(G)$ we have, for all $m,n\ge 0$,
	\begin{enumerate}
		\item \label{ab} $\psi_n(G)\le\psi(G)$
		\item \label{comm} $\psi_m(g)\psi_n(h)=\psi_n(h)\psi_m(g)$
		\item \label{sum}  $\psi_m+\psi_n=\psi_n+\psi_m$
		\item \label{comp} $(1-\psi)^n = 1-\psi_n$
		\item \label{recur} $\psi_{n+1}=\psi+\psi_n(1-\psi)$.
		\item \label{paren} $(\psi_m)_n = \psi_{mn}$.
	\end{enumerate}
\end{Lemma}
\begin{proof} 
First, (\ref{ab}) should be clear from Equation \ref{bin}, as is (\ref{comm}), and (\ref{sum}) is a special case of  (\ref{comm}). The formulation in (\ref{comp}) follows quickly from the definition of $\psi_n$. To show the recursive formula in (\ref{recur}), we have
	\begin{align*}
		\psi_{n+1}(g) &= (-(1-\psi)^{n+1}(g))g\\
		&=-((1-\psi_n)(1-\psi)(g))g\tag{by (\ref{comp})}\\
		&=\left((1-\psi_n)(g\psi(g^{-1})\right)^{-1}g\\
		&=\left(g\psi(g^{-1})\psi_n(g\psi(g^{-1})^{-1})\right)^{-1}g\\
		&=\psi_n(g\psi(g^{-1}))\psi(g)g^{-1}g\\
		&=\psi_n(1-\psi)+\psi\\
		&=\psi+\psi_n(1-\psi) \tag{by (\ref{comm})}.
	\end{align*} 
Finally, (\ref{paren}) follows from a quick computation:
\[(\psi_m)_n = -(1-\psi_m)^n + 1 = -((1-\psi)^m)^n+1 = (1-\psi)^{mn}+1=\psi_{mn}.\]
\end{proof}

\begin{remark}
	In light of the introduction to this section, Lemma \ref{lem} (\ref{paren}) should be reasonable. For example, the sixth map starting with $\psi$ should be the same as the third map starting with $\psi_2$ or the second map starting with $\psi_3$. 
\end{remark}

\section{Brace blocks and solutions to the {Y}ang-{B}axter equation}

Here we will introduce our main object of study: the brace block. We will show that $\psi\in\Ab(G)$ can be used to generate a brace block, and we will provide explicit solutions to the Yang-Baxter equation that are obtained from $\psi$.

\begin{definition}
	A {\it brace block} is a set $B$ together with  binary operations $\{\circ_n : n\ge 0\}$ such that $B_{m,n}:=(B,\circ_m,\circ_n)$ is a brace for all $m,n\ge 0$.
\end{definition}

Note that $B_{n,n}$ is the trivial brace on $(B,\circ_n)$. 

As mentioned in the introduction, the braces in a brace block are necessarily bi-skew, a concept developed by Childs in \cite{Childs19}. A {\it bi-skew brace} is a set $B$ with binary operations $\cdot$ and $\circ$ such that both $(B,\cdot,\circ)$ and $(B,\circ,\cdot)$ are (skew left) braces. Given a brace block, both $(B,\circ_m,\circ_n)$ and $(B,\circ_n,\circ_m)$ are braces, hence $(B,\circ_m,\circ_n)$ is bi-skew.

\begin{example}
	Any brace $(B,\cdot,\circ)$ can be made into a brace block with $x\circ_0 y= x\cdot y,\;x\circ_n y = x\circ y$ for all $x,y\in B, \;n\ge 1$. Since each brace is the trivial brace on $(B,\cdot)$ we will call this a {\it trivial brace block}.
\end{example} 

\begin{example}
	Suppose $(B,\cdot,\circ)$ is a nontrivial bi-skew brace. We can then form a nontrivial brace block with
	\[x\circ_n y = \begin{cases} x\cdot y & n \text{ even} \\ x\circ y & n \text{ odd} \end{cases}.\]
\end{example}

Starting with a (nonabelian) group $G=(G,\cdot)$ and a $\psi\in\Ab(G)$ we will construct a brace block. We start by constructing the necessary binary operations. Let $\psi\in\Ab(G)$. For each $n\ge 0$ define a binary operation $\circ_n$ on $G$  by
\[g\circ_{n} h = g\psi_n(g^{-1})h\psi_n(g),\;g,h\in G.\]
Notice that $\circ_0$ is the usual group operation in $G$, and $\circ_1$ is the operation denoted by $\circ$ in \cite{Koch20} since $\psi_1=\psi$. We will frequently write $\cdot$ and $\circ$ for $\circ_0$ and $\circ_1$ respectively. Since $\psi_n\in\Ab(G)$ we have, by \cite[Th. 1]{Koch20},

\begin{Lemma}
	For all $n\ge 0$, $(G,\circ_n)$ is a group.
\end{Lemma}

For future reference, the identity in $\G{n}$ is $1_G$, the identity in $(G,\cdot)$, and the inverse to $g\in(G,\circ_n)$ is $\psi_n(g)g^{-1}\psi_n(g^{-1})$, as can be readily checked.

We also have the following recursive formulation.

\begin{proposition}
	Let $\psi\in\Ab(G)$. Then for all $n\ge 0$ we have
	\[g\circ_{n+1} h = ((g\psi(g^{-1}))\circ_{n} h)\psi(g), \;g,h\in G.\]
\end{proposition}

\begin{proof}
	Lemma \ref{lem} (\ref{recur}) gives $\psi_{n+1}(g)=\psi(g)\psi_n(1-\psi)(g)$ for all $g\in G$. Applying this we get
	\begin{align*}
	g\circ_{n+1} h &= g\psi_{n+1}(g^{-1})h\psi_{n+1}(g)\\
	&=  g\psi_n(g\psi(g^{-1}))^{-1}\psi(g^{-1}) h\psi(g)\psi_n(g\psi(g^{-1}))\\
	&= g\psi(g^{-1})\psi_n(g\psi(g^{-1}))^{-1} h\psi_n(g\psi(g^{-1}))\psi(g) \tag{Lemma \ref{lem} (\ref{comm}),(\ref{sum})}\\
	&= ((g\psi(g^{-1}))\circ_{n} h)\psi(g).
\end{align*}
\end{proof}

From this we get
\begin{corollary}\label{abdone}
	Let $\psi\in \Ab(G)$. If $(G,\circ_n)$ is abelian, then $\G{n+1}=\G{n}$, hence $\G{m}=\G{n}$ for all $m\ge n$.
\end{corollary}

\begin{proof}
	Suppose $(G,\circ_n)$ is abelian. Then for all $g,h\in G$ we have
	\begin{align*}
		g\circ_{n+1} h &= ((g\psi(g^{-1}))\circ_{n} h)\psi(g)\\
		&=(h\circ_n g\psi(g^{-1}))\psi(g)\\
		&=h\psi_n(h^{-1})g\psi(g^{-1})\psi_n(h)\psi(g)\\
		&=h\psi_n(h^{-1})g\psi_n(h)\tag{Lemma \ref{lem} (\ref{comm})}\\
		&=h\circ_n g\\
		&=g\circ_n h.
	\end{align*}
Thus the group operations are identical, and $\G{n+1}=\G{n}$.
\end{proof}

Having established this collection of groups given by $\psi\in\Ab(G)$ we arrive at our main result, which states that any pair of the groups constructed above form a brace.
\begin{theorem}\label{main}
	Let $\psi\in\Ab(G)$. Then for all $m,n\ge 0$ we have $(G,\circ_m,\circ_n)$ is a brace.
\end{theorem}

%\begin{remark} This is a big generalization from the Fauxmaha talk, where I show (along with a few other cases) that $(G,\circ_{n-1},\circ_n)$ is a brace.
%	\end{remark}

\begin{proof}
	We simply need to show that $(G,\circ_m,\circ_n)$ satisfies the brace relation for all $m,n\ge 0$, that is, 
	\[g\circ_{n} (h\circ_mk) = (g\circ_n h)\circ_m \tilde g \circ_m (g\circ_m k) \] for all $g,h,k\in G$, where $\tilde g = \psi_m(g)g^{-1}\psi_m(g^{-1})$ is the inverse to $g$ in $(G,\circ_m)$. 
	
Recall that $\psi_m$ and $\psi_n$ are constant on conjugacy classes. We have
	\begin{align*}
	(g\circ_n h)\circ_m \tilde g \circ_m (g\circ_m k) &= (g\psi_n(g^{-1})h\psi_n(g))\circ_m \psi_m(g)g^{-1}\psi_m(g^{-1}) \circ_m (g\psi_n(g^{-1})k\psi_n(g))\\
&=	\Big((g\psi_n(g^{-1})h\psi_n(g))\psi_m(h^{-1}g^{-1}) \psi_m(g)g^{-1}\psi_m(g^{-1}) \psi_m(gh)  \Big)\circ_m (g\psi_n(g^{-1})k\psi_n(g))\\
&= \Big((g\psi_n(g^{-1})h\psi_n(g))\psi_m(h^{-1}) g^{-1} \psi_m(h)  \Big)\circ_m (g\psi_n(g^{-1})k\psi_n(g))\\
&=\big((g\psi_n(g^{-1})h\psi_n(g))\psi_m(h^{-1}) g^{-1} \psi_m(h)  \big)\psi_m(h^{-1}) (g\psi_n(g^{-1})k\psi_n(g))\psi_m(h)\\
&=(g\psi_n(g^{-1})h\psi_n(g))\psi_m(h^{-1})\psi_n(g^{-1})k\psi_n(g)\psi_m(h)\\
&=g\psi_n(g^{-1})h\psi_m(h^{-1})\psi_n(g)\psi_n(g^{-1})k\psi_m(h)\psi_n(g)\tag{Lemma 1.2 (\ref{comm})}\\
&=g\psi_n(g^{-1})h\psi_m(h^{-1})k\psi_m(h)\psi_n(g)\\
&= g\circ_{n} (h\circ_mk).
	\end{align*} 
Thus the brace relation is satisfied and we are done.
\end{proof}

From this we immediately obtain the following.
\begin{corollary}
	Let $\psi\in\Ab(G)$. Then $\psi$ gives a brace block.
\end{corollary}

In theory we can construct an unlimited number of binary operations using $\psi$, but of course only a finite number of them will be distinct. We have the following result, which is closely related to \cite[Prop. 3.3]{Koch20}.

\begin{proposition} \label{same}
	With the notation as above, the operations $\circ_m$ and $\circ_n$ agree if and only if $(\psi_m-\psi_n)(G)\subset Z(G)$.
\end{proposition}
\begin{proof}
	Suppose $g\circ_mh = g\circ_n h$ for all $g,h\in G$. Then
	\begin{align*}
	g\psi_m(g^{-1})h\psi_m(g) &=g\psi_n(g^{-1})h\psi_n(g) \\
	\psi_m(g^{-1})h\psi_m(g)&=\psi_n(g^{-1})h\psi_n(g) \\
	h &= \psi_m(g)\psi_n(g^{-1}) h \psi_n(g)\psi_m(g^{-1}),
	\end{align*}
	hence $\psi_m(g)\psi_n(g^{-1})=(\psi_m-\psi_n)(g)\in Z(G)$. The converse is trivial.
\end{proof}

As a result, we see patterns emerging in the binary operations.

\begin{corollary}
	If $g\circ_m h = g\circ_n h$ for all $g,h\in G$, then $g\circ_{m+1} h = g\circ_{n+1} h$ for all $g,h\in G$.
\end{corollary}

\begin{proof}
	It suffices to show that if $(\psi_m-\psi_n)\subset Z(G)$ then $(\psi_{m+1}-\psi_{n+1})\subset Z(G)$. Using Lemma \ref{lem} (\ref{comm}), (\ref{recur}) and the right distributive property in the near-ring we have
	\[\psi_{m+1}-\psi_{n+1} = \psi+\psi_m(1-\psi)-\psi-\psi_n(1-\psi) = (\psi_m-\psi_n)(1-\psi),\]
	and clearly $(\psi_m-\psi_n)(1-\psi)(G)\subset (\psi_m-\psi_n)(G)\subset Z(G)$.
\end{proof}

For future reference, we also note the following.

\begin{proposition}
	Let $\psi\in\Ab(G)$. Then $(1-\psi)(G)\le G$. Furthermore, $(G,\circ)$ is abelian if and only if $((1-\psi)(G),\cdot)$ is abelian.
\end{proposition}

\begin{proof}
	To show that $(1-\psi)(G)\le G$ it suffices to show that $(1-\psi)(G)$ is closed under the usual operation on $G$, which follows from the identity
	\[\big((1-\psi)(g)\big)\big((1-\psi)(h)\big)= (1-\psi)(g\psi(g^{-1})h\psi(g)),\] 
	which can be readily verified. 
	Now suppose $g\circ h = h\circ g$. Then
	\begin{align*} 
	g\psi(g^{-1})h\psi(g) &= h\psi(h^{-1})g\psi(h) \\
	g\psi(g^{-1})h\psi(h^{-1}gh) &= h\psi(h^{-1})g\psi(g^{-1}hg)\\
	g\psi(g^{-1})h\psi(h^{-1})\psi(gh) &= h\psi(h^{-1})g\psi(g^{-1})\psi(gh)\\
	g\psi(g^{-1})h\psi(h^{-1}) &= h\psi(h^{-1})g\psi(g^{-1}).
	\end{align*}
	Thus, $(G,\circ)$ is abelian if and only if $(g\psi(g^{-1}))(h\psi(h^{-1})) = (h\psi(h^{-1}))(g\psi(g^{-1}))$, that is, $(1-\psi)(G)$ is abelian.
\end{proof}

The above will be most useful when applied to $\psi_n$.

\begin{corollary}
	Let $\psi\in\Ab(G)$. Then $(1-\psi_n)(G)\le G$. Furthermore, $(G,\circ_n)$ is abelian if and only if $((1-\psi)^n(G),\cdot)$ is abelian.
\end{corollary}

\begin{proof} Apply the proposition to $\psi_n$, and recall that $1-\psi_n=(1-\psi)^n$. \end{proof}

\begin{example}
	Let $G=D_4=\gen{r,s:r^4=s^2=rsrs=1_G}$ be the dihedral group of order $8$. Let $\psi:G\to G$ be given by $\psi(r)=1,\;\psi(s)=r^2s$. It is easy to show that $\psi\in\Ab(G)$ (see also \cite[\S 6]{Koch20}). Since $(1-\psi)(r)=r$ and $(1-\psi)(s)=r^2$ we see that $(1-\psi)(G)=\gen{r}$ is abelian. Thus, $(G,\circ)$ is abelian.
\end{example}

\begin{example} Let $G=\mathrm{Aff}(\mathbb{F}_5)$, the affine group of the finite field with five elements. Then $G\cong C_5\rtimes \Aut(C_5)$ where $\Aut(C_5)$ acts on $C_5$ in the obvious way. If we let $C_5=\gen{g}$ and let $\alpha\in\Aut(C_5)$ be given by $\alpha(g)=g^2$ then 
	\[G=\gen{g,\alpha: g^5=\alpha^4 = g\alpha g^2\alpha^3=1_G}.\]
	Define $\psi:G\to G$ by $\psi(g)=1_G,\;\psi(\alpha)=\alpha^{-1}$. Then $\psi$ satisfies the relations above, and since $\psi(G)=\gen{\alpha}$ it follows that $\psi\in\Ab(G)$. We have
	\[(1-\psi)(g) = g,\; (1-\psi)(\alpha) = \alpha^2,\]
	hence $(1-\psi)(G)=\gen{g,\alpha^2}$ which is nonabelian since $\alpha^2 g = g^{-1}\alpha$, hence $(G,\circ)$ is nonabelian. (In fact, $(1-\psi)(G)\cong D_5$, the dihedral group of order $10$.) However,
	\[(1-\psi)^2(g) = g,\; (1-\psi)^2(\alpha) = (1-\psi)(\alpha^2) = 1_G\]
	so $\G{2}$ is abelian. It is not hard to show that $\G{2}\cong C_5\times C_4\cong C_{20}$.
	\end{example}

By applying the techniques of obtaining solutions from braces, we obtain:
\begin{theorem}
	Let $\psi\in\Ab(G)$. Then for all $m,n$ we have the following solutions to the Yang-Baxter equation:
	\begin{align*}
		R_{m,n}(g,h)&=\Big(\psi_m(g)\psi_n(g^{-1})h\psi_n(g)\psi_m(g^{-1}),\psi_m(g)\psi_n(g^{-1}h)h^{-1}\psi_n(g)\psi_m(g^{-1})g\psi_n(g^{-1})h\psi_n(gh^{-1})\Big)\\
		R_{m,n}'(g,h)&=\Big(g\psi_n(g^{-1})h\psi_n(g)\psi_m(h^{-1})g^{-1}\psi_m(h),\psi_n(h)\psi_m(h^{-1})g\psi_m(h)\psi_n(h^{-1})\Big)\\
	\end{align*}
\end{theorem}

The proof is a straightforward computation using Equations (\ref{one}) and (\ref{two}). In the case $m=0,\;n=1$ one can quickly recover the solutions $R_{1,\psi}$ and $R_{2,\psi}$ from \cite[Cor. 5.4]{Koch20}, formulas which also appear in \cite[Th. 5.1]{KochStordyTruman20} where $\psi$ is subject to a ``fixed point free'' condition (as explained below).

Note that if $m=n$ then we have $R_{n,n}=R'_{n,n}$.

\section{Preliminary Examples}

The motivation of the construction of brace blocks was to construct large families of braces emanating from a single abelian map $\psi$. Unfortunately, many simple examples of abelian maps found in \cite{Childs13}, \cite{KochStordyTruman20}, and \cite{Koch20} do not give large families. The maps considered in \cite{Childs13} and \cite{KochStordyTruman20} are subject to an additional condition: that of being fixed point free. An endomorphism $\psi:G\to G$ is said to be {\it fixed point free} if $\psi(g)=g$ if and only if $g=1_G$.

\begin{example}
	Let $G=D_n=\gen{r,s:r^n=s^2=rsrs=1}$. Any $\psi\in G$ must send $r$ and $s$ to elements of order dividing $2$. Thus, $2\psi = 0$. We will consider three cases.
	
	The first case is when $\psi=0$ is the trivial map. Of course, then $\psi_n = -(1-0)^n+1 = 0$ is also trivial, and we get a trivial brace block.
	
	Now suppose $\psi$ is fixed point free and nontrivial. Then, by \cite[\S 6]{Childs13}, $\psi(G)$ is a group of order 2, say $\psi(G)=\{1_G,x\}$. Since $\psi$ is fixed point free, $\psi(x)=1_G$, hence $\psi(\psi(g))=1_G$ for all $g\in G$, i.e., $\psi^2= 0$. Thus, for $k\ge 2$ we have
	\[\psi_k=\sum_{i=1}^k (-1)^{i-1}\binom ki \psi^{i} = k\psi=\begin{cases} \psi & k\text{ odd} \\ 0 & k\text{ even}\end{cases}. \]
Furthermore, both $(G,\cdot,\cdot)$ and $(G,\circ,\circ)$ are the trivial brace on $G$, hence $(G,\cdot,\cdot)=(G,\circ,\circ)$. It is known (see, e.g., \cite[Ex. 3.4]{KochTruman20b}) that a trivial brace is not isomorphic to a non-trivial brace; furthermore  we will see that $(G,\circ,\cdot)\cong(G,\cdot,\circ)$ (see Corollary \ref{fpfblah}). Consequently, this brace block has two nonisomorphic braces, namely $(G,\cdot,\cdot)$ and $(G,\cdot,\circ)$. Each provides two solutions to the Yang-Baxter equation, giving us four solutions in total.
	
	Now suppose $\psi$ has fixed points. Again, $\psi(G)$ is a group of order $2$, as shown in \cite[\S 6]{Koch20}. If we write $\psi(G)=\{1,x\}$ then $x$ must be the fixed point. Thus, $\psi^2(g)=\psi(g)$ for all $g\in G$ so $\psi^k=\psi$ for all $k\ge 1$. Thus,
	\[\psi_k=\sum_{i=1}^k (-1)^{i-1}\binom ki \psi^{i} = \left(\sum_{i=1}^k (-1)^{i-1}\binom ki\right)\psi=\psi. \]
	 Therefore, $\Bee{m}{n}=(G,\circ,\circ)$ for $m,n\ge 1$. By \cite[\S6]{Koch20} we know that $(G,\circ)\not\cong D_n$--it is, in fact, either $C_n\times C_2$ or $D_{n/2}\times C_2$ depending on the parity of $n$ and the choice of $x$--so we have four nonisomorphic braces: $(G,\cdot,\cdot),\;(G,\cdot,\circ),\;(G,\circ,\cdot)$, and $(G,\circ,\circ)$. If $(G,\circ)\cong D_{n/2}\times C_2$ then each gives two solutions to the Yang-Baxter equation, giving us a total of $8$ solutions, whereas if $(G,\circ)\cong C_n\times C_2$ we obtain $6$ solutions.

\end{example}

\begin{example}
	Let $G=S_n,\;n\ge 5$. Here each abelian map sends all even permutations to the identity and all odd permutations to an element of order $2$, say $\xi$. As above we have $2\psi=0$. If $\xi\in A_n$ then $\psi$ is fixed point free and clearly $\psi^2=0$. If $\xi\notin A_n$ then $\psi(G)=\gen{\xi}$ and $\psi^2=\psi$. 
	
	Thus we get the same two cases as above. In the fixed point free case we get the two braces $(G,\cdot,\cdot)$ and $(G,\cdot,\circ)$, and in the case with fixed points we get $(G,\cdot,\cdot),\;(G,\cdot,\circ),\;(G,\circ,\cdot)$, and $(G,\circ,\circ)$. In the latter case, $(G,\circ)\cong A_n\times C_2$ hence we will always get $8$ solutions.
\end{example}

\begin{example}\label{split}
	Suppose $G=HK= H\rtimes K$ with $K$ abelian. Then $\psi(hk)=k$ is an abelian map. Since
$\psi^2(hk) = k\text{ and } \psi((hk)^2)=k^2$.
	we get $\psi^2=\psi$ and $\psi_k=\psi$ as above. Our brace block produces braces $(G,\cdot,\cdot),\;(G,\cdot,\circ),\;(G,\circ,\cdot)$, and $(G,\circ,\circ)$, and as $(G,\circ)\cong H\times K$ we get $6$ or $8$ solutions depending on whether $H$ is abelian. 
\end{example}

We conclude this section with a closer look at the case where $\psi\in \Ab(G)$ is fixed point free. First, we note:

\begin{proposition}
	Let $\psi\in\Ab(G)$. Then $\psi$ is fixed point free if and only if $\psi_n$ is fixed point free for all $n\ge 1$.
\end{proposition}

\begin{proof}
	First, suppose $\psi(g)=g$ for some $g\in G,\;g\ne 1_G$. Note that $g^{-1}$ is also a fixed point of $\psi$. Then $(1-\psi)(g^{-1})=1_G$ and
	\[\psi_n(g)=(-(1-\psi)^n+1)(g) =\left((1-\psi)^n(g^{-1})\right)g = g\] 
	and hence $g$ is a fixed point of $\psi_n$ for all $n$.
	
	Conversely, suppose $\psi$ is fixed point free. Let $n$ be the smallest positive integer such that $\psi_n$ has nontrivial fixed points, say $\psi_{n}(g)=g$ for $g\ne 1_G$. By Lemma \ref{lem}, (\ref{recur}) we have
	\[
		g=\psi_{n}(g)=(\psi+\psi_{n-1}(1-\psi))(g)
		=\psi(g)\psi_{n-1}(g\psi(g^{-1})),
	\]
	which we may rewrite as
	\[\psi(g^{-1})g = \psi_{n-1}(g\psi(g^{-1})) = \psi_{n-1}(\psi(g^{-1})g),\]
	and so $\psi(g^{-1})g$ is a fixed point of $\psi_{n-1}$. But as $\psi_{n-1}$ is assumed to have no nontrivial fixed points it follows that $\psi(g^{-1})g=1_G$, i.e., $g=\psi(g)$, hence $g=1_G$, a contradiction. Thus, $\psi_n$ is fixed point free for all $n\ge 1$.
\end{proof}

The next result we will apply to fixed point free maps, however we state it more generally.

\begin{proposition}
	Let $\psi\in \Ab(G)$, and for $n\ge 1$ let $\phi=(1-\psi)\in\Map(G)$. Then $\phi:(G,\circ_n)\to (G,\circ_{n-1})$ is a homomorphism.
\end{proposition}

\begin{proof}
	We shall prove this by induction on $n$. Since 
	\begin{align*}
		\phi(g\circ h) &= (g\circ h)\psi(g\circ h)^{-1}\\
		 &=g\psi(g^{-1})h\psi(g)\psi(g\psi(g^{-1})h\psi(g))\\
		 &=g\psi(g^{-1})h\psi(g)\psi(h^{-1}g^{-1})\\
		 & = g\psi(g^{-1})h\psi(h^{-1})\\
		 &= \phi(g)\cdot\phi(g)
	\end{align*}
we see that $\phi:(G,\circ_1)\to (G,\circ_0)$ is a homomorphism. Now suppose $\phi: \G{k}\to\G{k-1}$ is a homomorphism. Using  Lemma \ref{lem} (\ref{comm}),(\ref{recur}) we get 
\begin{align*}
	\phi(g\circ_{k+1} h) &= (g\circ_{k+1} h)\psi(g\circ_{k+1} h)^{-1}\\
	&= g\psi_{k+1}(g^{-1})h\psi_{k+1}(g)\psi(h^{-1}g^{-1})\\
	&= g\psi(g^{-1})\psi_{k}(g^{-1}\psi(g))h\psi(g)\psi_{k}(g\psi(g^{-1}))\psi(g^{-1}h^{-1})\\
	&= g\psi(g^{-1})\psi_{k}((g\psi(g^{-1}))^{-1})h\psi(h^{-1})\psi_{k}(g\psi(g^{-1}))\\
	&= \phi(g)\circ_{k}\phi(h)
\end{align*}
and hence $\phi:\G{k+1}\to\G{k}$ is a homomorphism.
\end{proof}

Once the above is established, we quickly get

\begin{corollary}\label{fpfblah}
	Let $\psi\in\Ab(G)$ be fixed point free, and let $n>m\ge 0$. Then $(G,\circ_{m},\circ_n)\cong (G,\cdot,\circ_{n-m})$.
\end{corollary}

\begin{proof}
	The map $\phi: G\to G$ as defined above is bijective if and only if $\psi$ is fixed point free \cite[Lemma 10.1.1]{Gorenstein68}; hence $\phi^{n-m}:\Bee{m}{n}\to (G\cdot,\circ_{n-m})$ is an isomorphism in this case.
\end{proof}

We will see in the next section that one can sometimes use fixed point free abelian maps to generate relatively large brace blocks.

\begin{example}
	Let $n$ be a odd integer, and let $D_n$ denote the dihedral group of order $2n$. Consider the group $D_n\times D_n$, presented as 
	\[G=\gen{r,s:r^n=s^2=rsrs=1_G}\times \gen{t,u : t^n=u^2=tutu=1_G}.\]
	Define $\psi:G\to G$ by $\psi(r)=\psi(t)=1_G,\;\psi(s)=u,\;\psi(u)=s$. Then $\psi\in\Ab(G)$.
	Since $su\in\psi(G)$ and $\psi(G)$ is abelian, for all $g\in G$ we have
	\[
	g\circ su = g\psi(g^{-1})su\psi(g) = gsu,\;su\circ g = (su)(us)g(us) = gsu,
	\]
	hence $su\in Z(G,\circ)$. Since $Z(G,\cdot)=Z(D_n)\times Z(D_n)$ is trivial we get that $(G,\circ)\not\cong (G,\cdot)$. Also, $(G,\circ)$ is nonabelian since $r\circ u = ru,\;u\circ r = r^{-1}u$. In fact, it can be shown that $(G,\circ)=\gen{us}\times ((\gen{r}\times \gen{t})\rtimes \gen{u})\cong C_2\times ((C_n\times C_n)\rtimes C_2)$ where the semidirect product arises from the map $C_2\to\Aut(C_n\times C_n)$ sending the nontrivial element to the inverse map. 
	
	Now we compute $\psi_2:G\to G$. It is easy to see that $\psi_2(r)=\psi_2(t)=1_G$; furthermore
	\[\psi_2(s)=(2\psi-\psi^2)(s) = \psi(1_G)\psi^2(s) = s,\]
	and $\psi_2(u)=u$ similarly. Since $G = G_0G_1$ where $G_0=\ker\psi_2$ and $G_1$ is the subgroup of fixed points of $\psi_2$, by \cite[Prop. 6.3]{Koch20} we have that $(G,\circ_2)\cong G_0\times G_1\cong (C_n\times C_n)\times (C_2\times C_2) \cong C_{2n}\times C_{2n}$. (Of course, that $(G,\circ_2)$ is abelian also follows easily by observing $(1-\psi)^2(G)=\gen{r,t}\cong C_h\times C_h$.) 
	
	Since $(G,\circ_2)$ is abelian, we have that $(G,\circ_n)=(G,\circ_2)$ for all $n\ge 2$ by Corollary \ref{abdone}. Thus we have $9$ braces, namely $(G,\circ_i,\circ_j)$ for $0\le i,j\le 2$, all pairwise nonisomorphic. If $i\ne 2$ then $\Bee{i}{j}$ gives us two solutions to the Yang-Baxter equation, whereas $\Bee{1}{j}$ gives us one. In total, we have $15$ solutions.
\end{example}
 
\section{Semidirect products of two cyclic groups} 

Finally, we present a class of examples which provide brace blocks containing many nonisomorphic braces. We gratefully acknowledge Lindsay Childs for pointing out this class of examples.

We will follow the notation in \cite{ChildsCorradino07}. Pick an integer $h\ge 3$, and let $F(h,k,b)=\gen{s,t:s^h=t^k=tst^{-1}s^{-b}=1_G}$ where $k\mid \phi(h)$ and $b$ has multiplicative order $k\pmod h$. We also have $F(h,k,b^n)$ for any $n\ge 0$: while $b^n$ may not have multiplicative order $k$, there is some $c$ of multiplicative order $k$ such that $F(h,k,b^n)=F(h,k,c)$. As we shall see, it will be useful to refer to such groups using powers of $b$.

We start by addressing isomorphism questions among these groups.

\begin{Lemma}
	Let $h,k,b$ be as above. Let $n\in\mathbb Z$, and let $d=\gcd(k,n)$. Then $F(h,k,b^n)\cong F(h,k,b^d)$.
\end{Lemma}

\begin{proof}
	Pick $e\in\mathbb Z$ such that $en\equiv d\pmod k$: such an $e$ exists since $\gcd(n/d,k)=1$. We define $\gamma:F(h,k,b^d)\to F(h,k,b^n)$ by $\gamma(s^ut^v)=s^ut^{ve}$. To see this is well-defined, observe that
	\[\gamma(ts)=\gamma(s^{b^d}t)=s^{b^d}t^e=s^{b^{en}}t^e=t^es=\gamma(t)\gamma(s).\]
	As $e$ is invertible mod $k$ it follows that $\gamma$ an isomorphism.
\end{proof}

When $h,k$ and $b$ are understood we will write $F_n=F(h,k,b^n)$ for brevity. %Notice that $F_n$ is abelian if and only if $k\mid n$.

\begin{Lemma}
	Let $n,d$ be as above, and suppose $h$ is prime. If $d\ne k$ then $|Z(F_n)|=d$; otherwise, $F_n$ is abelian.
\end{Lemma}

\begin{proof}
	In light of the previous result we may assume $n=d$. Suppose $s^ut^v\in Z(F_d)$. Then
	\[s^{u+b^{dv}}t=(s^ut^v)s = s(s^ut^v)=s^{u+1}t^v,\]
	from which it follows that $b^{dv}\equiv 1 \pmod k$, i.e., $k\mid dv$. Furthermore,
	\[ s^ut^{v+1} = (s^ut^v)t = t(s^ut^v) =s^{ub^d}t^{v+1},\]
	and hence $u\equiv ub^d\pmod h$. Since $h$ is prime, it follows that either $u=0$ or $b^d\equiv 1\pmod h$. But if $b^d\equiv 1\pmod h$ then $k\mid d$, so $k=d$ and $F_d=F_k\cong C_h\times C_k$. If $k\ne d$ then $u=0$. Thus,
	\[Z(F_d)=\begin{cases} \gen{t^{k/d}} & k < d\\ G & k=d \end{cases}.\] 
\end{proof}

Combining the two previous results gives us

\begin{proposition}
	Suppose $h$ is prime. then $F_m\cong F_n$ if and only if $\gcd(m,k)=\gcd(n,k)$.
\end{proposition}

For the remainder of this section, we assume $h$ is prime.

We will now construct brace blocks starting with $G=F_1$. Pick $j\in \mathbb Z$ and define $\psi:G\to G$ by $\psi(s)=1_G,\;\psi(t)=t^{1-j}$. This map clearly respects the relations in $G$, and as $\psi(G)\le \gen{t}\cong C_k$ we see that $\psi\in\Ab(G)$. Note that if $j\equiv 1\pmod k$ then we will get a trivial brace block.

We shall now explicitly compute $\psi_n$. Note that
\[(1-\psi)(s)=s\psi(s^{-1})=s,\;(1-\psi)(t)=t\psi(t^{-1})=t\cdot t^{j-1}=t^j.\]
Thus,
\begin{align*}
	\psi_n(s)&=(-(1-\psi)^n+1)(s)=s^{-1}\cdot s=1_G\\
	\psi_n(t)&=(-(1-\psi)^n+1)(t)= t^{-j^n} \cdot t=t^{1-j^n}.
\end{align*}

Next, we compute the group $(G,\circ_n)$ for each $n$. Clearly, $s\circ_n g=sg$ for all $g\in G$; in particular, $s\circ_n t = st$. Also,

\begin{align*}
	t\circ_n t &= t\psi_n(t^{-1})t\psi_n(t) = t\cdot t^{j^n-1}\cdot t \cdot t^{1-j^n} = t^2\\
	t\circ_n s &= t\psi_n(t^{-1})s\psi_n(y) = t\cdot t^{j^n-1}\cdot s \cdot t^{1-j^n} = s^{b^{j^n}}t = s^{b^{j^n}}\circ_n t.
\end{align*}
Thus, $(G,\circ_n)=\gen{s,t:s^{\circ_n h}=t^{\circ_n k}=1_G,\;t\circ_n s = s\circ_n t^{b^{j^n}}\circ t}$ where ``$\circ_{n} m$'' in the exponent is the application of $\circ_n$ to the element $m$ times. In other words, $(G,\circ_n)=F_{j^n}$.

Notice that $(1-\psi)^n(s)=s$ and $(1-\psi)^n(t)=t^{j^n}$. Thus, $\G{n}=F_{j^n}$ is abelian if and only if $t^{j^n}\in Z(G)$. In particular, if $k\mid j^n$ then $(G,\circ_n)$ is abelian.

\begin{example}
	Suppose $k$ is prime. Then $h\equiv 1 \pmod k$ and $G=F_1$ is the unique nonabelian group of order $hk$. The abelian maps on $G$ appear in \cite[Ex. 3.8]{Koch20} using slightly different notation. If we pick $j$ such that $k\mid j$ then $\G{j}\cong C_h\times C_k$ and we get a specific instance of Example \ref{split}. 
	
	For all other choices of $j$ we have $\gcd(j^n,k)=1$ for all $n$, hence $\G{n}\cong G$ for all $n$. By \cite[Prop. 8.17]{KochTruman20b} we see that $(G,\cdot,\circ_m)\cong(G,\cdot,\circ_n)$ if and only if $1-j^m\equiv 1-j^n\pmod k$, i.e., $j^m\equiv j^n\pmod k$. Thus, if $j$ has multiplicative order $\ell\pmod k$ then $\{(G,\cdot,\circ_n): 0\le n \le \ell-1\}$ consist of $\ell-1$ distinct braces (note that $(G,\cdot,\circ_{\ell})=(G,\cdot,\cdot)$). By Corollary \ref{fpfblah} we see that all other braces in this brace block are isomorphic to one of braces in this set. 
	
	Since $(G,\cdot)=F_1$ is nonabelian we get $2\ell$ solutions to the Yang-Baxter equation. In fact, if we pick $j$ to be a primitive root $\bmod\;k$ then we obtain all $2(k-1)$ solutions to the Yang-Baxter equation arising from a brace with both groups isomorphic to the metacyclic group. Indeed, by \cite{KochTruman20b} each brace with both groups metacyclic is either in the block above or is the opposite to a brace in the block above. Thus, we have constructed all solutions to the Yang-Baxter equation coming from braces with metacyclic group structures of order $hk$.
\end{example}

\begin{example}
	Let $N$ be a large integer. By Dirichlet's Theorem, there exist prime numbers of the form $\ell\cdot 2^N+1$; let $h$ be one such prime, and let $k=2^N$. Pick $j=2$, i.e., $\psi(t)=t^{-1}$. Then $(G,\circ_m)=F_{2^m}$ for all $m$, and since $\gcd(2^m,2^N)=2^{\min\{m,N\}}$ we get that $(G,\circ_m)\not\cong(G,\circ_n)$ for $0\le m < n \le N$ and $(G,\circ_m)\cong(G,\circ_N)\cong C_{\ell\cdot 2^N+1}\times C_{2^N}\cong C_{(\ell+1)\cdot 2^{N}+1}$ for all $m\ge N$. Thus, we get a brace block consisting of $N$ pairwise nonisomorphic groups. 
	
	It follows that our brace block contains $(N+1)^2$ different braces, of which $N^2$ consist of nonabelian groups (obtained by requiring $m,n\ne N$), giving $4N^2$ solutions to the Yang-Baxter equation; $N$ braces containing exactly one abelian group, giving another $2N$ solutions; and $1$ brace with both groups abelian, giving one more solution. In total, we get
	$4N^2+2N+1=(2N+1)^2$
	solutions to the Yang-Baxter equation. 
	
\end{example}

These examples show that there are brace blocks with an arbitrarily number of pairwise non-isomorphic groups, and that the number of possible solutions to the YBE using an abelian map is unbounded.

\bibliographystyle{alpha} 
\bibliography{../../MyRefs}
\end{document}